\documentclass[aop, preprint]{imsart}
\usepackage{amsmath,amsthm,amsfonts,amssymb,mathrsfs,bm,graphicx,stmaryrd,hyperref}
\usepackage[usenames]{color}

\newtheorem{theorem}{Theorem}[section]
\newtheorem{lemma}[theorem]{Lemma}
\newtheorem{corollary}[theorem]{Corollary}
\newtheorem{proposition}[theorem]{Proposition}

\newtheorem{property}{Property}[section]
\newtheorem{remark}[theorem]{Remark}

\newtheorem{definition}[theorem]{Definition}

\newcommand{\citeDOV}{{\sf DOV}}

\def\N{\mathbb{N}}
\def\Q{\mathbb{Q}}
\def\P{\mathbb{P}}
\def\Z{\mathbb{Z}}
\def\R{\mathbb{R}}
\def\S{\mathcal{S}}
\def\C{\mathcal{C}}
\def\B{\mathcal{B}}
\def\A{\mathcal{A}}
\def\F{\mathcal{F}}
\def\H{\mathcal{H}}
\def\J{\mathcal{J}}
\def \W{\mathrm{W}}

\def\cR{\mathcal{R}}
\def\cL{\mathcal{L}}

\def\E{\mathbb{E}}
\def \e{\varepsilon}
\def \epsilon{\varepsilon}

\def\l{\ell}
\DeclareMathOperator*{\argmax}{arg\,max}

\newcommand{\sourav}[1]{{\textcolor{black}{#1}}}

\begin{document}
\begin{frontmatter}
\title{Brownian absolute continuity of the KPZ fixed point with arbitrary initial condition}
\runtitle{Brownian absolute continuity of the KPZ fixed point}
\begin{aug}
\author[A]{\fnms{Sourav} \snm{Sarkar}\ead[label=e1]{ssarkar@math.toronto.edu}}
\and
\author[B]{\fnms{B\'alint} \snm{Vir\'ag}\ead[label=e2]{balint@math.toronto.edu}}
\address[A]{Department of Mathematics,
University of Toronto,
Toronto, ON, Canada.
\printead{e1}}

\address[B]{Departments of Mathematics and Statistics,
University of Toronto,
Toronto, ON, Canada.
\printead{e2}}
\end{aug}

\begin{abstract}We show that the law of the KPZ fixed point starting from arbitrary initial condition is absolutely continuous with respect to the law of Brownian motion $B$ on every compact interval. In particular, the Airy$_1$ process is absolutely continuous with respect to $B$ on any compact interval.
\end{abstract}

\begin{keyword}[class=MSC2010]
\kwd{82B23}
\kwd{82C22}
\kwd{60H15}
\end{keyword}

\end{frontmatter}



\section{Introduction} In 1986, Kardar, Parisi and Zhang \cite{KPZ86} predicted universal scaling behavior for many planar random growth processes. A central object to describe the random growth models in the KPZ universality class is the Airy line ensemble. The parabolic Airy line ensemble is a random sequence of functions $\A_1> \A_2 > \ldots$. It was introduced by Pr\"ahofer and Spohn \cite{prahofer2002scale} in the version $\A_i(t) + t^2$, which is stationary in time, see also Corwin and Hammond \cite{CH}. The functions $A_i(t)$
are locally Brownian with
diffusion parameter $2$. The top line $\A_1$ is known as the parabolic Airy$_2$ process that appear as the limiting spatial fluctuation of random growth models starting from a single point. The Airy sheet and directed landscape were constructed by Dauvergne, Ortmann and Vir\'ag \cite{DOV} (\citeDOV\ in the sequel) as the scaling limits of one of the KPZ-class models, Brownian last passage percolation. The standard Airy sheet $\S:\R^2\mapsto \R$ is a random continuous function defined in terms of the Airy line ensemble such that $\S(0,\cdot)=\A_1(\cdot)$. The Airy sheet of scale $s$ is defined by
\[\S_s(x,y)=s\S(x/s^2,y/s^2)\,,\]
for any $s>0$.
 The construction of the directed landscape from the Airy sheet is analogous to that of the Brownian motion from Gaussian distribution. Let $\R^4_\uparrow=\{(x,s;y,t)\in \R^4: s<t\}$. The coordinates $x$ and $y$ can be thought of as spatial and the coordinates $s$ and $t$ as temporal. Then the directed landscape $\mathcal L: \R^4_\uparrow\mapsto\R$ is a random continuous function satisfying
the metric composition law
\[\cL(x,r;y,t)=\sup_{z\in \R}(\cL(x,r;z,s)+\cL(z,s;y,t))\,,\]
for all $(x,r,y,t)\in \R^4_\uparrow$ and all $s\in (r,t)$; and with the property that $\cL(\cdot, t;\cdot, t + s^3)$ are independent Airy sheets of scale $s$ for any set of
disjoint time intervals $(t,t+s^3)$.

On the other hand, all models in the KPZ universality class have an analogue of the
height function which is conjectured to converge at large time and small length scales under the KPZ $1:2:3$ scaling to a universal object $h_t(\cdot)$ called the KPZ fixed point. Matetski-Quastel-Remenik \cite{matetski2016kpz} construct the KPZ fixed point as a Markov process in $t$, and they show that it is a limit of the
height function evolution of the totally asymmetric simple exclusion process with arbitrary initial condition. \sourav{Later in \cite{NQR20} Nica-Quastel-Remenik proved the convergence of the Brownian last passage percolation to the KPZ fixed point. As the directed landscape was constructed earlier from Brownian last passage percolation in \citeDOV, this showed that the KPZ fixed point also admits the variational formula in terms of the directed landscape.} That is, the KPZ fixed point
can also be written in terms of the directed landscape and its initial condition $h_0:\mathbb R \to \mathbb R\cup \{-\infty\}$ as
\[h_t(y) = \sup_{x\in \R} (h_0(x) + \cL(x, 0; y, t))\,,\]
for all $y\in \R$. For the narrow wedge initial condition, $h_0(0)=0$ and $h_0(x)=-\infty$ elsewhere, $h_1(\cdot)=\A_1(\cdot)$ is the parabolic Airy$_2$ process. For $h_0\equiv 0$, the flat initial condition, $h_1(\cdot)$ is called the Airy$_1$ process. In this paper we show that the law of $h_t(\cdot)$ is absolutely continuous with respect to Brownian motion on any compact interval, for all initial conditions for which the solutions are finite everywhere at a given positive time.

Let $\C[a,b]$ denote the space of all continuous functions $f:[a,b]\mapsto \R$ and $\C_0[a,b]$ denote the space of all continuous functions $f:[a,b]\mapsto \R$ with $f(a)=0$, with the topologies of uniform convergence. We shall denote the space of all continuous functions in $\R$ as $\C$ and $\C^\N$ denotes the space of all continuous functions from $\R\times \N\mapsto \R$ with the topology of uniform convergence on compact sets.

\begin{definition}\label{d:fspace} A function $f:\R\to \R\cup \{-\infty\}$ will be called a $t$-{\bf finitary initial condition} if $f(x)\neq -\infty$ for some $x$, $f$ is bounded from above on any compact interval and
\[\frac{f(x)-x^2/t}{|x|}\to -\infty\]
as $|x|\to \infty$.
  A finitary initial condition (for some, or equivalently, all $t$) will be said to be {\bf compactly defined} if $f(x)=-\infty$ outside a compact set.
\end{definition}
The name comes from the fact that $h_t(x)$ is finite for all $x\in \mathbb R$ if and only if the initial condition is $t$-finitary. A strong version of this is shown in Proposition \ref{l:conditionh0}.

For example, a function $f\not \equiv -\infty$, satisfying $f(x)\leq C+x^2/t-|x|\log(1+|x|)$ for some $C>0$, is a $t$-finitary initial condition.  The upper semicontinuous initial conditions for which the KPZ fixed point was established in \cite{matetski2016kpz} are $t$-finitary for all $t>0$.

Fix a diffusion parameter $\sigma^2$; in this paper $\sigma^2=2$. We call a random function in $F$ in $\C$ {\bf Brownian on compacts} if for all $y_1<y_2$, the law of the $\C_0[y_1,y_2]$-random variable
$$y\mapsto F(y)-F(y_1)\,,$$
is absolutely continuous with respect to the law of a Brownian motion starting from $(y_1,0)$ with diffusion parameter $\sigma^2$ on $[y_1,y_2]$.
The following theorem is the main result of this paper.
\begin{theorem}\label{t:gen} Let $t>0$; then for any $t$-finitary initial condition $h_0$ the random function
\[h_t(y) = \sup_{x\in \R} (h_0(x) + \cL(x, 0; y, t))\,,\]
is Brownian on compacts.
\end{theorem}

\begin{remark}
In particular, the Airy$_1$ process is Brownian on compacts. Hence for any compact set $K$, the process has a unique point of maximum on $K$ a.s.\
\end{remark}

\subsection{A partial review of some recent works} The Brownian nature of the KPZ fixed point in general, and the Airy$_2$ process in particular, has been a subject of intense research in recent times. We provide a brief review here; the interested reader is referred to Calvert, Hammond and Hegde \cite{calvert2019brownian} and the references therein for a more elaborate account.

The locally Brownian nature of the Airy$_2$ process has been previously shown in different formulations. One relatively weak version is to consider local limits of the Airy$_2$ process; that is, to show that $\e^{-1/2}(\A(x+\e)-\A(x))$ for a given $x\in\R$ converges in law to a Brownian motion as $\e\to 0$. This was shown in H\"{a}gg \cite{H08}, Cator and Pimentel \cite{CP15}, Quastel and Remenik \cite{QR13}. In fact, \cite{QR13} also establishes H\"older $1/2-$ continuity of the Airy$_2$ and Airy$_1$ processes.  More recently, the H\"older $1/2-$ continuity and the locally Brownian nature (in terms of convergence of the finite dimensional distributions of $\e^{-1/2}(\A(x+\e)-\A(x))$) were established in \cite[Theorem 4.13 and Theorem 4.14]{matetski2016kpz}. Such H\"older continuity results and local limits in the space of continuous functions for certain initial conditions have also been established in Pimentel \cite{Pim18} and {more recently in Pimentel \cite{Pim19}}.

A stronger notion of the
locally Brownian nature is absolute continuity with respect to Brownian motion on compact intervals, {which we call Brownian on compacts. This was shown in Quastel and Remenik \cite{QR11} for the Hopf-Cole solution at time $t>0$ of the
KPZ equation, a random growth model in the KPZ universality class, starting from special initial conditions which are sums of Brownian motions and Lipschitz functions.} That the Airy$_2$ process is Brownian on compacts was first proved in \cite{CH} using the \textit{Brownian Gibbs property}. The result has been recently extended in \cite{calvert2019brownian} (see also Hammond \cite{Ham19a}) to show that the Radon Nikodym derivative is in $L^p$ space for all $p\in (1,\infty)$. However, for general initial conditions, no such absolute continuity result was known. The best known result in this direction was a form of Brownian regularity, called \textit{patchwork quilt of Brownian fabrics}, established in Hammond \cite{Ham19d} and \cite{calvert2019brownian}. Roughly the result states that the KPZ fixed point $h(\cdot)$ on a unit interval is divided into a random number of subintervals, the patches, with random boundary points in such a way that the restriction of the profile to each patch is absolutely continuous with respect to a Brownian motion with Radon-Nikodym derivative in $L^p$ for all $p<3$. Conjecture $1.3$ in \cite{Ham19d} asks for establishing this with a single patch and Radon-Nikodym derivative is in $L^p$ for all $p>0$, a problem which remains open.


Our result in Theorem \ref{t:gen} of this paper corresponds to establishing absolute continuity with a \textit{single} patch.
As stated in pages 9-10 in \cite{Ham19d}, the reduction from a random number of patches to a single patch and removal of all the stitch points require novel ideas. Indeed, except for the use of the Brownian Gibbs property, our proof of Theorem \ref{t:gen} in this paper proceeds along very different lines than that in \cite{Ham19d} and \cite{calvert2019brownian}.

\sourav{We believe that the arguments in this paper can be upgraded to show that the Radon-Nikodym derivative of the KPZ fixed point in a compact interval with respect to a Brownian motion is in $L^p$ for all $p>0$. However, some new ideas and estimates will be required to get this strong control on the Radon-Nikodym derivative.}

\subsection{Organization of the paper} In Section \ref{s:prelim} we introduce the background material and the relevant results that will be used in the paper. Section \ref{s:geo} is devoted to the construction and properties of infinite geodesics in the Airy line ensemble. In Section \ref{absbrow} we prove that Brownian last passage percolation is absolutely continuous with respect to Brownian motion on compact intervals. Finally in Section \ref{s:proof} we prove Theorem \ref{t:gen}. We prove Proposition \ref{l:conditionh0} in Section \ref{s:last}.

\subsection{A few words about the proofs} Our proofs are probabilistic in nature and rest on the geometry of the underlying Airy sheet and the directed landscape. In \citeDOV, the Airy sheet was constructed as a last passage percolation in the Airy line ensemble. Based on that construction and armed with the Brownian Gibbs property, the law of the KPZ fixed point \eqref{e:defh2} can be described in terms of last passage percolation across the Airy line ensemble. We can compare this to Brownian last passage percolation. That is, Theorem \ref{t:gen} boils down to showing that for each $k\in \N$, and $k$ independent Brownian motions starting from $k$ arbitrary points, the top line of an extension of the Brownian \textit{melon} is absolutely continuous with respect to Brownian motion. This is proved by showing it for $k=2$ and using the metric composition law of last passage percolation and a simple induction argument.

\section{Preliminaries}\label{s:prelim} In this section we recall the relevant results and background material from \citeDOV\ that we will need in this paper. For a possibly finite sequence of continuous functions $f=(f_1,f_2,\ldots)$ with domain $\R$ and coordinates $x\leq y$ and $m\leq \l$, we define a \textbf{path} from $(x,\l)$ to $(y,m)$ as a nonincreasing function $\pi:[x,y]\mapsto \N$ which is cadlag
on $(x,y)$ and satisfies $\pi(x)=\l$ and $\pi(y)=m$. The paths are nonincreasing in order to fit in with the natural indexing of the Airy line ensemble. We define the \textbf{length} of $\pi$ as the sum of increments of $f$ along the path $\pi$. That is, for each $m\le i<\l$, if $t_i$ denotes the point where $\pi$ jumps from $f_{i+1}$ to $f_i$, then the length of $\pi$ is given by
\[\l(\pi)=f_m(y)-f_m(t_m)+\sum_{i=m+1}^{\l-1} (f_i(t_{i-1})-f_i(t_i))+f_\l(t_{\l-1})-f_\l(x)\,.\]
For $x\leq y\in \R$ and $m<\l\in \N$, we define the \textbf{last passage value} of $f$ from $(x,\l)$ to $(y,m)$ by
\[f[(x,\l)\to (y,m)]:=\sup_\pi \l(\pi)\,,\]
where the supremum is taken over all paths $\pi$ from $(x,\l)$ to $(y,m)$. Any path $\pi$ for which $\l(\pi)$ is the last passage value between its endpoints is called a \textbf{last passage path} or a \textbf{geodesic}.

We say that a point $(x,t)$ lies \textbf{along} a path $\pi:[s,r]\mapsto \N$ if $t\in [s,r]$ and
\[\lim_{r\to t-}\pi(r)\geq x\geq \lim_{r\to t+}\pi(r)\,.\]
In other words, if the graph of $\pi$ is connected at its jumps by vertical lines, then $(x,t)$ will lie on this connected version of the graph.

We also define the gap process $g=g(f)$ by $g_i=f_i-f_{i+1}$. Then for any path $\pi$ from $(x,\l)$ to $(y,m)$, its length can also be written in terms of the gap process, that is,
\begin{equation}\label{e:gap}
\l(\pi)=f_m(y)-f_\l(x)-\sum_{i=m}^{\l-1}g_i(t_i)\,,
\end{equation}
where, as before, the times $t_i$ correspond to the jumps of the path $\pi$ from $f_{i+1}$ to $f_i$. Thus the last passage value can also be interpreted as a difference of endpoints minus a minimal sum of gaps. Observe that for any function $h:\R\mapsto \R$, if $f'_i=f_i+h$ for all $i=1,2,\ldots$, then the gap process and the minimal sum of gaps remain unchanged. Hence a geodesic in $f$ is also a geodesic in $f'$.

The following \textbf{metric composition law} is enjoyed by last passage  and is quoted here from \citeDOV, Lemma $3.2$.
\begin{property}[Metric composition law]\label{p:metrcomp}Let $x\leq y\in \R$ and $m<\l\in \N$. Then for any $k\in \{m,m+1,\ldots, \l\}$, we have
\begin{eqnarray*}
f[(x,\l)\to (y,m)]&=& \sup_{z\in [x,y]}(f[(x,\l)\to (z,k)]+f[(z,k)\to (y,m)])\\
&=& \sup_{z\in [x,y]}(f[(x,\l)\to (z,k)]+f[(z,k-1)\to (y,m)])\,.
\end{eqnarray*}
Similarly, for any $z\in [x,y]$,
\[f[(x,\l)\to (y,m)]= \sup_{k\in  \{m,m+1,\ldots, \l\}}(f[(x,\l)\to (z,k)]+f[(z,k)\to (y,m)])\,.\]
This implies a (reverse) triangle inequality for last passage value. For any $x\leq z\leq y$ and $k\in \{m,m+1,\ldots, \l\}$,
\begin{equation}\label{e:triangleineq}
f[(x,\l)\to (y,m)]\geq f[(x,\l)\to (z,k)]+f[(z,k)\to (y,m)]\,.
\end{equation}
\end{property}
For $x\leq y\in \R$ and $m<\l\in \N$, there exist last passage paths $\pi^-,\pi^+$  between $(x,\l)$ and $(y,m)$ such that for any last passage path $\pi$ between $(x,\l)$ and $(y,m)$ and any $z\in [x,y]$, we have
\[\pi^-(z)\leq \pi(z)\leq \pi^+(z)\,.\]
We call $\pi^-$ the \textbf{leftmost last passage path} and $\pi^+$ the \textbf{rightmost last passage path} (see \citeDOV, Lemma $3.5$). The following property is derived from \citeDOV, Proposition $3.7$.
\begin{property}[Ordering of geodesics]\label{p:porder}Let $x_1\leq x_2\leq y_1\leq y_2\in \R$ and $m<\l\in \N$. Let $\pi_i^+$ denote the rightmost last passage path from $(x_i,\l)$ to $(y_i,m)$ for $i=1,2$. Then for all $s\in [x_2,y_1]$,
\[\pi_1(s)\leq \pi_2(s)\,.\]
\end{property}

\subsection{Pitman transform}Let $\C^2_+$ be the space of continuous functions $f=(f_1,f_2)$ where $f_i:[0,\infty)\mapsto \R$ for $i=1,2$. For $f\in \C^2_+$, we define $\mathrm Wf=(\W f_1,\W f_2)\in \C^2_+$, the Pitman transform of $f$ as follows. For $x<y\in [0,\infty)$, define the maximal gap size
\[{G(f_1,f_2)(x,y):=\max\{\max_{s\in [x,y]}(f_2(s)-f_1(s)),0\}\,.}\]
Then define
\begin{equation}\label{e:pitmantrans}
\W f_1(t)=f_1(t)+G(f_1,f_2)(0,t)\,,
\end{equation}
\[\W f_2(t)=f_2(t)-G(f_1,f_2)(0,t)\,,\]
for all $t\in [0,\infty)$. Informally, the graph of $\W f_1$ represents the reflection of the graph of $f_1$ off $f_2$. This is also called the \emph{Skorohod reflection}. Note that this definition is slightly different from that in \citeDOV, Section $4$, in order to accommodate functions not starting from the origin.

Since $G(f_1,f_2)(0,t)\geq 0$, hence
\[\W f_1(t)\geq f_1(t)\,,\qquad \mbox{ and }\qquad \W f_2(t)\leq f_2(t)\,,\]
and since $G(f_1,f_2)(0,t)\geq f_2(t)-f_1(t)$, hence
\[\W f_1(t)\geq f_2(t)\,,\qquad \mbox{ and }\qquad \W f_2(t)\leq f_1(t)\,,\]
for all $t\in [0,\infty)$.

We can write $\W f_1$ in terms of the last passage values. To see this, first note that
\[\W f_1(t)=f_1(t)+\max\{\max_{s\in [0,t]}(f_2(s)-f_1(s)),0\}=\max\{f_1(t),\max_{s\in [0,t]}(f_2(s)+f_1(t)-f_1(s))\}\,.\]
Since $f_1(t)=f_1(0)+f[(0,1)\to(t,1)]$ and $\max_{s\in [0,t]}(f_2(s)+f_1(t)-f_1(s))\}=f_2(0)+f[(0,2)\to(t,1)]$, hence
\begin{equation}\label{e:ptr}
\W f_1(t)=\max_{i=1,2}\{f_i(0)+f[(0,i)\to(t,1)]\}\,,
\end{equation}
for all $t\in [0,\infty)$.

For $f_1(0)=f_2(0)=0$, $\W f$ is also called the \textbf{melon} of $f$. In this case, clearly for all $t\in [0,\infty)$,
\[\W f_1(t)=f[(0,2)\to(t,1)]\,.\]
A special case of this for Brownian motions is particularly interesting. Let $B=(B_1,B_2)$ be two independent standard Brownian motions (started from the origin). Let $\hat{B}=(\hat{B}_1,\hat{B}_2)$ be the law of $B$ conditioned (in the sense of Doob) never to collide. That is, let $h(x)=(x_1-x_2)$ for $x=(x_1,x_2)\in \R^2$. Then $h(x)$ is a positive harmonic function on the set $\{x\in \R^2: x_1>x_2\}$ and $\hat{B}$ is the Doob $h$-transform of $B$. Then $\W B$ has the same law as $\hat{B}$. This is essentially equivalent to Pitmans' $2M-B$ Theorem in Pitman \cite{Pi75}. A generalization of this result for $n$ Brownian motions was proven in O'Connell and Yor \cite{o2002representation}.
\subsection{Brownian Gibbs property} The following property enjoyed by the parabolic Airy line ensemble $\A$ was established in \cite[Theorem $3.1$]{CH}. For any $k,\l\in \{0,1,2,\ldots\}$ and $a<b\in \R$, let $\F$ denote the sigma algebra generated by the set
 \[\{\A_i(x): (i,x)\notin\{k+1,k+2,\ldots,k+\l\}\times(a,b)\}\,.\]
 Then the conditional distribution of $\A_{\big|\{k+1,2,\ldots,k+\l\}\times[a,b]}$ given $\F$ is equal to the law of $\l$ independent Brownian bridges $B_1,B_2,\ldots, B_\l:[a,b]\mapsto \R$ with diffusion parameter $2$ such that $B_i(a)=\A_{k+i}(0)$ and $B_i(b)=\A_{k+i}(b)$, for all $i=1,2,\ldots, \l$, conditioned on the event
 \[\A_k(r)>B_1(r)>B_2(r)>\ldots>B_\l(r)>\A_{k+\l+1}(r)\]
for all $r\in [a,b]$. For this to hold for $k=0$, we set $\A_0\equiv\infty$.

\subsection{Properties of the Airy sheet} The Airy sheet $\S:\R^2\mapsto \R$ is a random continuous function that was constructed in \citeDOV\ in terms of the parabolic Airy line ensemble $\A$. We recall some of its properties here that will be used throughout this paper (see Definition $1.2$ and Lemma $9.1$ of \citeDOV).
\begin{enumerate}
\item[(i)] \textit{Translation invariance}: $\S$ has the same law as $\S(\cdot+t,\cdot+t)$ for all $t\in \R$.
\item[(ii)] $\S(0,\cdot)=\A_1(\cdot)$.
\item[(iii)] $\S(x,y)\overset{d}{=} \S(y,x)$ and $\S(x,y)\overset{d}{=} \S(-x,-y)$. 
\end{enumerate}

\section{Geodesics in the Airy line ensemble}\label{s:geo}
In this section we will construct infinite geodesics between $x\in \Q^+$ and $y\in \Q$ and study their properties. But before proceeding, we need to recall a couple of theorems.

  Let $\W B^n$, the Brownian $n$-melon, be described as $n$ independent standard Brownian motions conditioned never to collide (in the sense of the Doob $h$-transform with $h(x)=\prod_{i<j}(x_i-x_j)$ for $x\in \R^n$, which is a positive harmonic function on the Weil chamber $\{x\in \R^n: x_1>x_2>\ldots>x_n\}$). The following theorem was proved in many parts, see Pr{\"a}hofer and Spohn \cite{prahofer2002scale}, Johansson \cite{johansson2003discrete}, Adler and Van Moerbeke \cite{adler2005pdes} and Corwin and Hammond \cite{CH}.

\begin{theorem}\label{t:conv}Define the rescaled melon $A^n=(A_1^n,\ldots, A_n^n)$ by
\[A_i^n(y)=n^{1/6}\left ((\W B^n)_i(1+2yn^{-1/3})-2\sqrt{n}-2yn^{1/6}\right)\,.\]
Then $A^n$ converges in law to the (parabolic) Airy line ensemble $\A=(\A_1,\A_2,\ldots)\in \C^\N$ with respect to product of uniform-on-compact topology on $\C^\N$.
\end{theorem}

\sourav{The next definition of the Airy sheet is quoted from \citeDOV, Definition $1.2$.}
\begin{definition}
The Airy sheet $\S$ can be coupled with the (parabolic) Airy line ensemble $\A$ so that $\S(0,\cdot)=\A_1(\cdot)$ and almost surely
 for all $(x,y,z)\in \Q^+\times \Q^2$, there exists a random integer $K_{x,y,z}$ such that for all $k\ge K_{x,y,z}$ \begin{equation}\label{e:defS}
\A[x_k\to (z,1)]-\A[x_k\to (y,1)]=\S(x,z)-\S(x,y)\,,
\end{equation}
where $x_k=(-\sqrt{k/2x},k)$.
\end{definition}
We shall use this coupling of the Airy sheet throughout the paper.

For $x\leq y\in \R$ and $\l\in \N$, we shall denote the rightmost geodesic between $(x,\l)$ and $(y,1)$ in the Airy line ensemble $\A$ by $\pi[(x,\l)\to y]$.

Next we define the infinite geodesics in the Airy line ensemble.

\begin{definition}\label{d:infgeo}For any $x\in\Q^+$ and $y\in \Q$ with $x_k=(-\sqrt{k/2x},k)$, we define the geodesic $\pi[x\to y]$ as the almost sure pointwise limit of $\pi[x_k\to y]$ as $k\to \infty$, whenever the limit exists. We define the length of the geodesic $\pi[x\to y]$ as $\S(x,y)$.
\end{definition}
\begin{figure}
\includegraphics[width=5in]{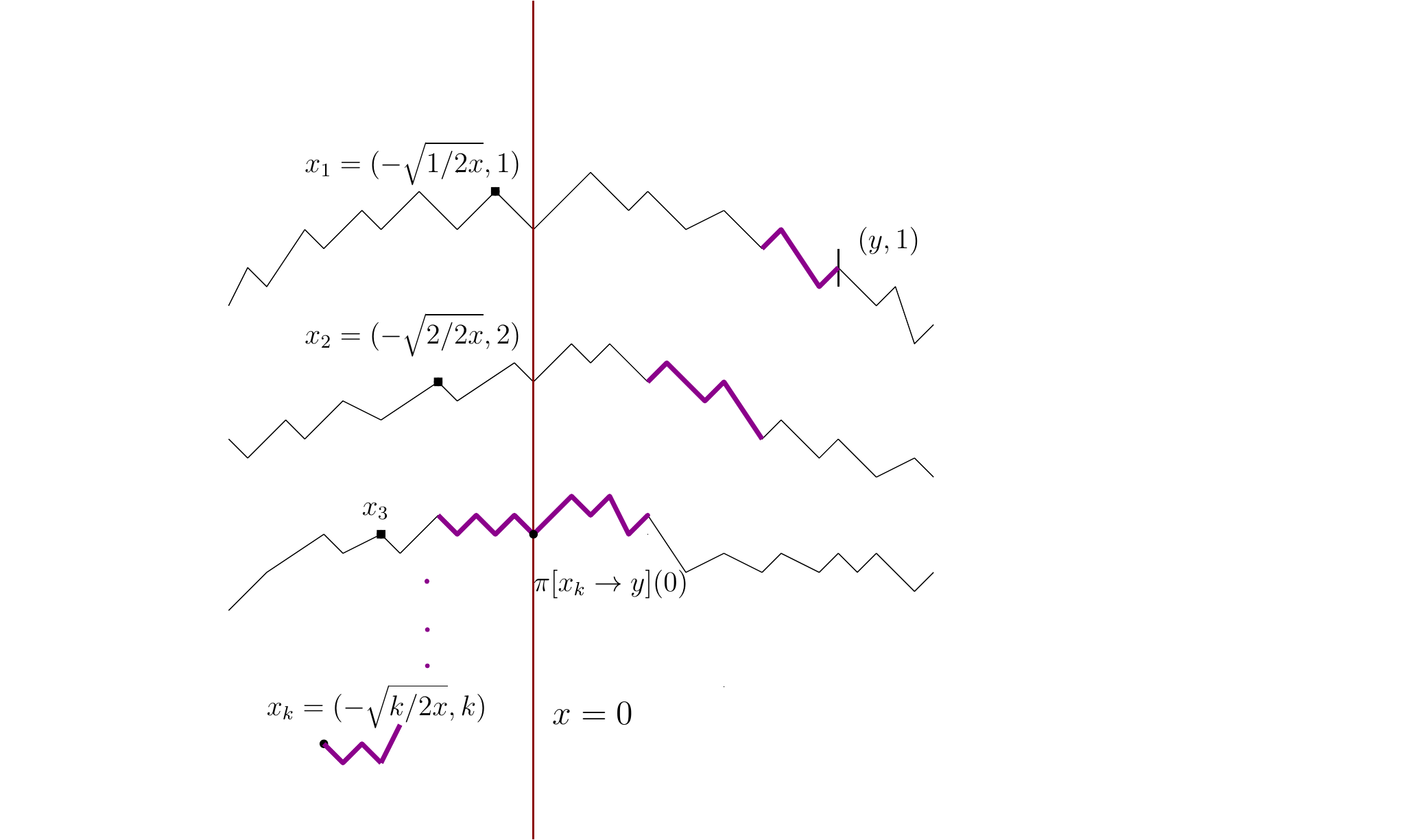}
\caption{The lines in the above figure denote the parabolic Airy line ensemble $(\A_1,\A_2,\ldots)$; we only show the first three lines here. The thick purple path here denotes the last passage path or geodesic from $x_k=(-\sqrt{k/2x},k)$ to $(y,1)$ and is denoted by $\pi[x_k\to y]$, which is a non-increasing function from $[-\sqrt{k/2x},y]$ to $\{1,2,\ldots, k\}$. For example, in the figure, $\pi[x_k\to y](0)=3$. Almost surely, the paths $\pi[x_k\to y]$ have a pointwise limit as $k\to \infty$, which is denoted by $\pi[x\to y]$.}
\label{F.1}
\end{figure}
\sourav{Thus $\pi[x\to y]$ is a non-increasing cadlag function from $(-\infty,y]$ to $\N$ with $\pi[x\to y](y)=1$. For any $w\in (-\infty,y]$, $\pi[x\to y](w)=r\in \N$ if and only if the path $\pi[x\to y]$ at the point $w$ passes through the $r$-th line of the Airy line ensemble, see Figure \ref{F.1}.} 
The following lemma guarantees that such geodesics exist almost surely for all rationals $x,y$.
\begin{lemma}\label{l:infgeo}Almost surely for all $(x,y)\in \Q^+\times \Q$, the geodesic $\pi[x\to y]$ exists. Moreover, for any $y\in \Q$,
\begin{equation}\label{e:pitranslate}
\pi[x\to y](\cdot+y)\overset{d}{=}\pi[x\to 0](\cdot)\,.
\end{equation}
Also for any $x\in \Q^+, y<z\in \Q$, almost surely there exists a random $T\leq y \in \R$ and $K\in \N$ (depending on $x,y,z$) such that
\begin{equation}\label{e:loc}
\pi[x\to y](T)=\pi[x\to z](T)=\pi[x_k\to y](T)=\pi[x_k\to z](T)\,,
\end{equation}
for all $k\geq K$. That is, the paths $\pi[x_k\to y]$, $\pi[x_k\to z]$, $\pi[x\to y]$ and $\pi[x\to z]$ intersect for all large $k$. Moreover, for all $t\geq T$ and $k\geq K$,
\begin{equation}\label{e:unique}
\pi[x\to y](t)=\pi[x_k\to y](t)\qquad \mbox{ and }\qquad \pi[x\to z](t)=\pi[x_k\to z](t)\,.
\end{equation}
And if the common value in \eqref{e:loc} is denoted by $d(T)$, then almost surely for all $y\leq y_1\leq y_2\leq z$,
\begin{equation}\label{e:wow}
\S(x,y_1)-\S(x,y_2)=\A[(T,d(T))\to (y_1,1)]-\A[(T,d(T))\to (y_2,1)]\,.
\end{equation}
\end{lemma}
\begin{proof} For any $(x,y)\in \Q^+\times\Q$, we first construct a path, which we call $\pi[x,y]$ to begin with. Afterwards we show that $\pi[x_k\to y]$ converges to $\pi[x,y]$ pointwise almost surely, so that $\pi[x\to y]$ exists and equals $\pi[x,y]$ almost surely.

The points $1., 2., 3.$ below are quoted from points $2.,3.,4.$ of the itemized list in page $37$ of \citeDOV. For $n\in \N$, let
\[\bar{x}=2xn^{-1/3}\,,\qquad \mbox{ and } \qquad \hat{y}=1+2yn^{-1/3}\,.\]
Let $\gamma_n:=\pi\{\bar{x}\to \hat{y}\}_n$ be the rightmost last passage path between $\bar{x}$ and $\hat{y}$ in the melon $\W B^n$. For $n\in \N$ and $k\in \{1,2,\ldots, n\}$, let $Z_k^n(x,y)$ be the supremum of $w$ so that $(w,k)$ lies along $\gamma_n$. Then, by \citeDOV, Lemma $7.1$, it follows that for each $k\in \N$, the sequence $\{Z^n_k(x,y)\}_n$ is tight. By Skorohod representation theorem, there exists a coupling of the process $\W B^n$ and $\A$ and a subsequence, such that along that subsequence, almost surely
\begin{itemize}
\item[1.] The melon $\W B^n$ in the scaling in Theorem \ref{t:conv} converges to the Airy line ensemble $\A$ uniformly on compact sets in $\Z\times \R$.

\item[2.]We have for all $(x,y)\in \Q^+\times \Q$,
\[Z_k^n(x,y)\to Z_k(x,y) \qquad \mbox{ for all } k\in \N.\]
Moreover, as $k\to \infty$,
\begin{equation}\label{e:zkconv}
Z_k(x,y)/\sqrt{k} \to -1/\sqrt{2x}\,.
\end{equation}

 \item[3.] For every triple$(x,y,z)\in \Q^+\times \Q^2$ with $y\le z$, there exist random points $X_1<x<X_2 \in \Q^+$ such that
\[\pi\{\bar{X_1}\to \hat{y}\}_n \qquad \mbox{ and } \qquad \pi\{\bar{X_2}\to \hat{z}\}_n\]
are not disjoint for all large enough $n$.
\end{itemize}

Then we define $\pi[x, y]:(-\infty, y]\to \Z$ as the non-increasing cadlag function given by
\[\pi[x,y](t)=\min\{k\in \N:Z_{k+1}(x,y)\leq t\}\]
for all $t\in (-\infty,y]$. Thus, $Z_k(x,y)$ is the supremum of $w$ so that $(w,k)$ lies along  $\pi[x, y]$. The path $\pi[x,y]$ is an almost sure pointwise limit of $\gamma_n$ over the subsequence. Moreover, Property $1$ above guarantees that $\pi[x, y]$ is a rightmost last passage path when restricted to any compact interval.

\sourav{Now fix any $x\in \Q^+, y<z\in Q$ as in the statement of Lemma \ref{l:infgeo}}. Let $X_1<x<X_2\in \Q^+$ be as in Property $3.$ above, that is, $\pi\{\bar{X_1}\to \hat{y}\}_n$ and $\pi\{\bar{X_2}\to \hat{z}\}_n$ are not disjoint for all large enough $n$. Because of \eqref{e:zkconv}, there exists a random $K\in \N$ such that
\[-\sqrt{k/2x}\in (Z_k(X_1,y), Z_k(X_2,z))\,.\]
for all $k\geq K$. From here \eqref{e:loc} follows by passing to the limit as $n\to \infty$ and using the ordering of the geodesics, with $\pi[x,y], \pi[x,z]$ in place of $\pi[x\to y],\pi[x\to z]$ respectively. The last equation \eqref{e:wow} follows by the ordering of geodesics. (The argument here is exactly the same as that of \citeDOV,  Lemma $8.5$., which can be seen for more details.)

Observe that $\pi[x,y]$ and $\pi[x_k\to y]$ restricted to $[T,y]$ are both rightmost geodesics between $(T,d(T))$ and $(y,1)$. Hence
for all $t\geq T$ and $k\geq K$,
\begin{equation}\label{e:pix,y}
\pi[x, y](t)=\pi[x_k\to y](t)\qquad \mbox{ and }\qquad \pi[x, z](t)=\pi[x_k\to z](t)\,.
\end{equation}

Next we claim that for $(x,y)\in \Q^+\times\Q$, almost surely for all $r\in \Z; r< y$, there exists a random $K\in \N$ (depending on $x,y,r$) such that for all $t\in [r,y]$ and all $k\geq K$,
\[\pi[x, y](t)=\pi[x_k\to y](t)\,.\]
Indeed, by \eqref{e:pix,y} with $x\in \Q^+$ and $r<y$, we have that there exists a random $T\leq r$ and $K\in \N$ such that for all $t\in [T,y]$ and all $k\geq K$,
\[\pi[x, y](t)=\pi[x_k\to y](t)\,.\]
Since $T\leq r$ and $[r,y]\subseteq [T,y]$, the claim follows. Hence almost surely $\pi[x_k\to y]$ converges $\pi[x, y]$ pointwise; thus $\pi[x\to y]=\pi[x,y]$ almost surely. Similarly, $\pi[x\to z]=\pi[x,z]$ almost surely. This, together with the last paragraph, proves \eqref{e:loc}. Also \eqref{e:unique} follows because both $\pi[x\to y]$ and $\pi[x_k\to y]$ are rightmost geodesics between $[T,y]$, a similar statement holds for $\pi[x\to z]$ and $\pi[x_k\to z]$.

Finally we prove \eqref{e:pitranslate}. \sourav{Fix any $x\in \Q^+, y\in \Q$ as in the statement of Lemma \ref{l:infgeo}}. Let $X_1<x<X_2\in \Q^+$ be as in Property $3$ above. Because of \eqref{e:zkconv} again, almost surely there exists a random $K\in \N$ such that
\[-\sqrt{k/2x}+y\in (Z_k(X_1,y), Z_k(X_2,y))\, \qquad \mbox{and }-\sqrt{k/2x}\in (Z_k(X_1,y), Z_k(X_2,y))\,.\]
for all $k\geq K$. This is because by \eqref{e:zkconv} and Property $3$, almost surely for all large enough $k$, $Z_k(X_1,y)/\sqrt{k}<-1/\sqrt{2x}+y/\sqrt{k}< Z_k(X_2,y)/\sqrt{k}$. Hence, arguing as before, we have $\pi[x\to y]$ is also the almost sure pointwise limit of $\pi[x_k+(y,0)\to y]$, the rightmost geodesic between $x_k+(y,0)=(-\sqrt{k/2x}+y,k)$ and $(y,1)$. Now since a geodesic in the parabolic Airy line ensemble $\A$ is also a geodesic in the stationary Airy line ensemble $(\A_i(t)+t^2:i=1,2,\ldots)$ (see the discussion following \eqref{e:gap}), hence
\[\pi[x_k+(y,0)\to y](\cdot+y)\overset{d}{=}\pi[x_k\to 0](\cdot)\,,\]
for all $k\in \N$. Taking the limit as $k\to \infty$ on both sides, \eqref{e:pitranslate} follows.
\end{proof}


In the next lemma we show that the geodesics $\pi[x\to y]$ also enjoy the ordering of geodesics property.

\begin{lemma}\label{l:infporder}Let $x\leq z\in \Q^+$ and $y_1\leq y_2\in \Q$. Then almost surely for all $s\leq y_1$,
\[\pi[x\to y_1](s)\leq \pi[z\to y_2](s)\,.\]
\end{lemma}
\begin{proof}By Lemma \ref{l:infgeo} there exists $K$ such that for all $k\geq K$
\[\pi[x\to y_1](s)=\pi[x_k\to y_1](s)\,, \qquad \mbox{ and }\qquad \pi[z\to y_2](s)=\pi[z_k\to y_2](s)\,.\]
Since $x\leq z$, we have $-\sqrt{k/2x}\leq -\sqrt{k/2z}$. Hence the lemma follows by applying Property \ref{p:porder} to the geodesics $\pi[x_k\to y_1]$ and $\pi[z_k\to y_2]$.
\end{proof}
Recall that for $\l\in \N$, $Z_\l(x,y)$ is the supremum of $z$ so that $(z,\l)$ lies along $\pi[x\to y]$.
\begin{lemma}\label{l:useful}Fix any $x\in \Q^+$ and $\l \in \N$. Then
\[\lim_{y\to \infty, y\in \Q}\P(Z_\l(x,y)\leq 0)=0\,.\]
\end{lemma}
Because of ordering of geodesics, $\{Z_\l(x,y)> 0\}\subseteq \{Z_\l(x,y')> 0\}$ for any $y'>y\in \Q$. Thus, the above lemma implies that there exists a random $Y\in \Q^+$ such that almost surely
\[Z_\l(x,Y)>0\,.\]
\begin{proof}This follows immediately since by \eqref{e:pitranslate} of Lemma \ref{l:infgeo},
\[Z_\l(x,y)\overset{d}{=}y+Z_\l(x,0)\,.\]
\end{proof}

Finally, we need to define $\A[x\to (0,\l)]$ for $x\in \Q^+$ as a random function in $x$ for any $\l\in \N$ as follows. For $\l=1$, we define $\A[x\to (0,1)]=\S(x,0)$ and for all $\l>1,x\in \Q^+$, we define it as the almost sure limit of
\[\A[x_k\to (0,\l)]-\A[x_k\to (0,1)]+\S(x,0)\,,\]
as $k \to \infty$. For this definition to make sense, we need the above limit in $k$ to exist. This is the crux of the next theorem.

\begin{theorem}\label{t:ok} Almost surely for all $x\in \Q^+$, there exists a random integer $K'_{x,\l}$ such that for all $k\geq K'_{x,\l}$
\begin{equation*}
\A[x_k\to (0,\l)]-\A[x_k\to (0,1)]
\end{equation*}
does not depend on $k$. We define $\A[x\to(0,\l)]$ as
\[\A[x_k\to (0,\l)]-\A[x_k\to (0,1)]+\S(x,0)\]
for any such $k\geq K'_{x,\l}$.
\end{theorem}

\begin{proof}We prove that the above statement holds for a fixed $x\in \Q^+$ almost surely. Its extension to all $x\in \Q^+$ is standard and holds almost surely.
To this end, fix any $x\in \Q^+$. Using Lemma \ref{l:useful}, we get a random $Y\in \Q^+$ such that $Z_\l(x,Y)>0$ almost surely. Then Lemma \ref{l:infgeo} guarantees that there exist $(T,d(T))$ and $K$ such that almost surely for all $k\geq K$, the paths $\pi[x\to 0] ,\pi[x\to Y], \pi[x_k\to 0]$ and $\pi[x_k\to Y]$ intersect at $(T,d(T))$. Since $T\leq 0$, and $Z_\l(x,Y)>0$,
\[d(T)>\l\,.\]
Since
\[Z_\l(x,0)\leq 0<Z_\l(x,Y)\,,\]
by ordering of geodesics, for all $k\geq K$, $\pi[x_k\to (0,\l)]$ also passes through $(T,d(T))$. Thus for all $k\geq K$,
\[\A[x_k\to (0,\l)]-\A[x_k\to (0,1)]=\A[(T,d(T))\to (0,\l)]-\A[(T,d(T))\to (0,1)]\,.\]
Hence the lemma follows.
\end{proof}

The next lemma provides a bound for $\A[x\to (0,\l)]$.
\begin{lemma}\label{l:boundA}For all $\l\in \N$ and $x\in \Q^+$, almost surely
\[\A[x\to (0,\l)]\leq \S(x,0)\,.\]
\end{lemma}

\begin{proof}
By Theorem \ref{t:ok}, for any $x\in \Q^+$, $\A[x\to(0,\l)]-\S(x,0)$ is defined as the almost surely limit as $k\to \infty$ of
\[\A[x_k\to (0,\l)]-\A[x_k\to (0,1)]\,. \]
Now by triangle inequality \eqref{e:triangleineq}, we have for every $k,\l\in \N$ with $k\geq \l$,
\[\A[x_k\to (0,1)]\geq \A[x_k\to (0,\l)]\,.\]
Thus $A[x\to(0,\l)]-\S(x,0)$ is an almost sure limit of nonpositive random variables, hence
\[\A[x\to(0,\l)]\leq \S(x,0)\]
almost surely.
\end{proof}

Next let $\mathcal F_-$ denote the $\sigma$-field generated by the negative time values of $\mathcal A$, that is
\[\mathcal F_-=\sigma\{\A_i(x): x\leq 0, i=1,2,\ldots\}\,.\]
Then we have the following lemma.
\begin{lemma}\label{insigma}
 For all $\l\in \N$ and $x\in \Q^+$,
 $$\A[x\to (0,\l)]\in \mathcal F_-\,.$$
\end{lemma}
\begin{proof}The proof follows in two steps.

Step 1. $\mathcal S(x,0)\in \mathcal F_-$.

Using translation invariance ($\S(\cdot+t,\cdot+t)$ has the same law as $\S$ for all $t\in \R$) we have
\[\S(x,\cdot)\overset{d}{=}\S(0,\cdot-x)=\A_1(\cdot-x)\,.\]
 Since $\A_1(z)+z^2$ is stationary and ergodic (see equation $(5.15)$ in \cite{prahofer2002scale}) for any fixed $x>0$ we have almost surely,
\begin{eqnarray}\label{e:sx0}
\S(x,0)&=&\lim_{m\to \infty} \frac{1}{m}\int_{-m}^0(\S(x,0)-(\S(x,z)+(z-x)^2))dz+\E\A_1(0)\nonumber\\
&=&\lim_{m\to \infty} \frac{1}{m}\int_{-m}^0((\S(x,0)-\S(x,z))-(z-x)^2)dz+\E\A_1(0)
\end{eqnarray}
Now for any $z<0$, $\S(x,0)-\S(x,z)\in \mathcal F_-$ because it is defined as the almost sure limit
\[\lim_{k\to \infty}(\A[x_k\to (0,1)]-\A[x_k\to(z,1)])\,,\]
where $x_k=(-\sqrt{k/(2x)},k)$, and clearly $\A[x_k\to (0,1)]-\A[x_k\to(z,1)]\in \F_-$ for any $z<0$. Thus $\S(x,0)\in \mathcal F_-$ from \eqref{e:sx0}.

Step 2. $\mathcal A[x\to(0,\l)]-\S(x,0)\in  \mathcal F_-$.

This is because by Theorem \ref{t:ok}, $\mathcal A[x\to(0,\l)]-\S(x,0)$ is obtained as an almost sure limit
$$
\lim_{k\to\infty}(\mathcal A[x_k\to(0,\l)]-\mathcal A[x_k\to(0,1)])
$$
with $x_k=(-\sqrt{k/(2x)},k)$, and clearly $\A[x_k\to(0,\l)]-\mathcal A[x_k\to(0,1)]\in \mathcal F_ -$.

From these two steps we have $\A[x\to (0,\l)]\in \F_-$.
\end{proof}

Finally we have the following lemma.

\begin{lemma}\label{l:decomp}Let $x_0>1$ and $y_0>1$. Let
\[L_0=\pi[x_0'\to y_0'](0)\,,\]
for some $x_0',y_0'\in \Q$ with $x_0'\geq x_0$ and $y_0'\geq y_0$. Then almost surely for all $x\in [1,x_0]\cap \Q$ and all $y\in[1,y_0]$,
\[\S(x,y)=\max_{\l\le L_0}(\A[x\to(0,\l)]+\A[(0,\l)\to (y,1)])\,.\]
\end{lemma}
\begin{proof}
By Lemma \ref{l:infporder}, almost surely for all $x\in [1,x_0]\cap \Q$,
\[\pi[x\to y_0'](0)\leq \pi[x_0'\to y_0'](0)=L_0\,.\]
Hence, by Definition \ref{d:infgeo} and Lemma \ref{l:infgeo}, almost surely for all $x\in [1,x_0]\cap \Q$ and all $k$ large enough (depending on $x$),
\[\pi[x_k\to y_0'](0)\leq L_0\,.\]
Thus by ordering of geodesics (Property \ref{p:porder}), almost surely for all $x\in [1,x_0]\cap \Q$ and all $k$ large enough (depending on $x$) and all $y\in [1,y_0]$,
\[1\leq \pi[x_k\to y](0)\leq \pi[x_k\to y_0'](0)\leq L_0\,.\]
Hence almost surely for all $x\in [1,x_0]\cap \Q$ and all $k$ large enough, depending on $x$ and all $y\in [1,y_0]$, by \eqref{e:wow} and metric composition law (Property \ref{p:metrcomp}), we have
\begin{eqnarray*}
\S(x,y)-\S(x,0)&=&\A[x_k\to(y,1)]-\A[x_k\to(0,1)]\\
&=&\sup_{\l\in \N}(\A[x_k\to(0,\l)]+\A[(0,\l)\to (y,1)])-\A[x_k\to(0,1)]
\end{eqnarray*}
Since $\pi[x_k\to y](0)\le L_0$, we have
\begin{eqnarray*}
\S(x,y)-\S(x,0)&=&\max_{\l\le L_0}(\A[x_k\to(0,\l)]+\A[(0,\l)\to (y,1)])-\A[x_k\to(0,1)]\\
&=&\max_{\l\le L_0}(\A[x_k\to(0,\l)]-\A[x_k\to(0,1)]+\A[(0,\l)\to (y,1)])\,.
\end{eqnarray*}
Thus almost surely for all $x\in [1,x_0]\cap \Q$, \sourav{taking} $k\ge \sup_{\l\le L_0}K'_{x,\l}$ where $K'_{x,\l}$ is as in Theorem \ref{t:ok}, we have by Theorem \ref{t:ok}, for all $y\in [1,y_0]$,
\begin{eqnarray*}
\S(x,y)-\S(x,0)&=&\max_{\l\le L_0}(\A[x\to(0,\l)]-\S(x,0)+\A[(0,\l)\to (y,1)])\\
&=&\max_{\l\le L_0}(\A[x\to(0,\l)]+\A[(0,\l)\to (y,1)])-\S(x,0)\,.\qedhere
\end{eqnarray*}
\end{proof}

\section{Absolute continuity of Brownian last passage percolation}\label{absbrow}
In this section, we show the absolute continuity with respect to Brownian motion of a certain Brownian last passage percolation. First we shall need the following lemma about the absolute continuity of the top line of the Pitman transform of two independent Brownian motions. Throughout this and the next section,
\[\mu\ll\nu\]
for two measures $\mu$ and $\nu$ will mean $\mu$ is absolutely continuous with respect to $\nu$.
\begin{lemma}\label{l:case2}Let $0<L<R$ and $b_1,b_2\in \R$. Let $B_1,B_2$ be two independent Brownian motions on $[0,\infty)$ starting from $(0,b_1)$ and $(0,b_2)$ respectively. Let $\mathrm{W}B_1(\cdot)$ denote the top line of the Pitman transform \eqref{e:pitmantrans} of $B_1$ with respect to $B_2$. Then the law of $\mathrm{W}B_1(t)$ restricted to $[L,R]$ is absolutely continuous with respect to that of a standard Brownian motion starting from $(0,0)$ restricted to $[L,R]$.
\end{lemma}

\begin{proof}It follows from the the formula for the Pitman transform that $(\mathrm{W}B_1(t), B_2(t))$ is a Markov process in $t$. Let $\mu_{b_1,b_2}=(\mu_{b_1,b_2}^{(1)},\mu_{b_1,b_2}^{(2)})$ be the law of $(\mathrm{W}B_1(\cdot), B_2(\cdot))$ restricted to $[L,\infty)$ and $\nu_{b_1,b_2}$ be the distribution of $(\mathrm{W}B_1(L),B_2(L))$, where $B_1, B_2$ start from $(0,b_1)$ and $(0,b_2)$ respectively. First observe that
\begin{equation}\label{e:absco}
\nu_{b_1,b_2}\ll \nu_{0,0}\,,
\end{equation}
since both are mutually absolutely continuous with respect to the product Lebesgue measure on the half-plane $\{(x,y)\in \R^2: x\ge y\}$.
Since $(\mathrm{W}B_1(t), B_2(t))$ is a Markov process, we have
\begin{equation}\label{e:abc}
\frac{d\mu_{b_1,b_2}}{d\mu_{0,0}}=\frac{d\nu_{b_1,b_2}}{d\nu_{0,0}} \qquad \mbox{ and }
\qquad \mu_{b_1,b_2}\ll \mu_{0,0}\,.
\end{equation}
Now $\mu_{0,0}^{(1)}$ is the law of the top line of a usual Brownian $2$-melon restricted to $[L,\infty)$, and the Brownian $2$-melon has the law of two independent standard Brownian motions conditioned not to intersect. Restricted to $[L,R]$, its law is absolutely continuous with respect to that of two independent Brownian motions. Thus its top line is absolutely continuous with respect to a standard Brownian motion restricted to $[L,R]$. This, together with \eqref{e:abc}, proves the lemma.
\end{proof}

We also have the following corollary of the above lemma.
\begin{corollary}\label{l:corcase2}Let $0<L<R$ and $b\in \R$. Let $B_1,B_2$ be two independent continuous processes on $[0,\infty)$ such that $B_1$ is a standard Brownian motion starting from $(0,b)$ and the law of $B_2$ restricted to $[\frac{L}{2},R]$ is absolutely continuous with respect to the law of a standard Brownian motion starting from $(0,0)$ restricted to $[\frac{L}{2},R]$. Let $\mathrm{W}B_1(\cdot)$ denote the top line of the Pitman transform of $B_1$ with respect to $B_2$, see \eqref{e:pitmantrans}. Then the law of $\mathrm{W}B_1(t)$ restricted to $[L,R]$ is absolutely continuous with respect to that of a standard Brownian motion starting from $(0,0)$ restricted to $[L,R]$.
\end{corollary}

\begin{proof} We prove this corollary in four steps.

\medskip
\noindent\textbf{ Step 1.} We first show that the law of $\mathrm{W}B_1(2L/3)$ given $B_2$ is absolutely continuous with respect to $\lambda$, the Lebesgue measure on $\R$. Indeed, by \eqref{e:pitmantrans}
$$
\mathrm{W}B_1(2L/3)=\max\{B_1(2L/3),\max_{s\in [0,2L/3]} (B_2(s)-B_1(s)+B_1(2L/3))\}\,.
$$
With $B_3(s):=-B_1(2L/3-s)+B_1(2L/3)$ we get by the time-reversal symmetry of Brownian motion increments that $B_3$ is a standard Brownian motion and $B_3(2L/3)=B_1(2L/3)-b$. Thus
\begin{equation}\label{e:corwb}
\mathrm{W}B_1(2L/3)=\max\{B_3(2L/3)+b,\max_{s\in [0,2L/3]} (B_2(2L/3-s)+B_3(s))\}\,.
\end{equation}
Since the $\argmax_{s\in [0,2L/3]}(B_2(2L/3-s)+B_3(s))=\argmax_{s\in [0,2L/3]}(-B_4(s)+B_3(s))$ where $B_4(s):=B_2(2L/3)-B_2(2L/3-s)$ and by the assumption of our corollary, $B_4(\cdot), 0\le s\le L/6$ is absolutely continuous with respect to a standard Brownian motion independent of $B_3$,
\begin{equation}\label{e:maxend}
\P\Big(\argmax_{s\in [0,2L/3]}(B_2(2L/3-s)+B_3(s))=0\Big)\le \P\Big(\argmax_{s\in [0,L/6]}(-B_4(s)+B_3(s))=0\Big)=0\,.
\end{equation}
Now, we define $B_5(s)=B_2(2L/3-s)$ for $s\in [0,2L/3)$ and $B_5(2L/3)=\max\{b,B_2(0)\}$. By \eqref{e:maxend}, the maximum of $B_5(s)+B_3(s)$ on $[0,2L/3]$ is almost surely attained at $s>0$. Thus by \eqref{e:corwb},
\[\mathrm{W}B_1(2L/3)=\max_{s\in (0,2L/3]} (B_5(s)+B_3(s))\,.\]
By our assumptions, $B_3$ and $B_5$ are independent and $B_3$ is a standard Brownian motion. Now we condition on $B_2$, then $B_5$ is a fixed bounded function. Then our claim follows if we can show that the maximum of a Brownian motion $B_3$ plus a fixed bounded function $f$ on $(0,2L/3]$ is absolutely continuous with respect to $\lambda$, where 
$f$ is such that the maximum of $B_3+f$ on $(0,2L/3]$ is almost surely attained. To this end, for any $n\in \N$ such that $1/n<2L/3$, let $\mu^f_n$ be the law of
\[\max_{s\in [1/n,2L/3]} (f(s)+B_3(s))=B_3(1/n)+\max_{s\in [1/n,2L/3]} (f(s)+B_3(s)-B_3(1/n))\,.\]
Because of the independence of $B_3(1/n+\cdot)-B_3(1/n)$ and $B_3(1/n)$ and the absolute continuity of the law of $B_3(1/n)$, which is Gaussian with mean $0$ and variance $1/n$, with respect to $\lambda$, conditioning on $B_3(1/n+\cdot)-B_3(1/n)$, we get that $\mu^f_n\ll \lambda$. Thus for any measurable set $A$ such that $\lambda(A)=0$, $\mu^f_n(A)=0$ for all $n\in \N, n>3/(2L)$. Hence
\[\P\Big(\max_{s\in (0,2L/3]} (f(s)+B_3(s))\in A\Big)\le \sum_{n\in \N, n>3/(2L)}\mu^f_n(A)=0\,.\]
That is, the law of $\max_{s\in (0,2L/3]} (f(s)+B_3(s))$ is absolutely continuous with respect to $\lambda$.

\medskip
\noindent{\textbf{Step 2.}} Let $\mu$ denote the measure of a standard Brownian motion on $[0,R-2L/3]$. Let $B_2(2L/3 +\cdot)$ denote the process $B_2(2L/3 +s)$ for $s\in [0, R-2L/3]$. Then we claim that the joint law of 
\[(WB_1(2L/3), B_2(2L/3), B_2(2L/3 +\cdot)-B_2(2L/3))\ll \lambda^2\times \mu\,.\] 
As the law of $B_2$ restricted to $[\frac{L}{2},R]$ is absolutely continuous with respect to the law of a standard Brownian motion starting from $(0,0)$ restricted to $[\frac{L}{2},R]$, in particular this implies that the joint law of $(B_2(2L/3), B_2(2L/3+\cdot)-B_2(2L/3))$ is absolutely continuous with respect to $\lambda\times \mu$. By Step 1, the conditional law of $WB_1(2L/3)$ given $(B_2(2L/3), B_2(2L/3+\cdot)-B_2(2L/3))$ is absolutely continuous with respect to $\lambda$. This proves the claim.

\medskip
\noindent{\textbf{Step 3.}} Let 
\begin{align*}
B_1'(t)&=B_1(t+2L/3)-B_1(2L/3)+ \mathrm{W}B_1(2L/3),\qquad \mbox{ for } t\in [0, R-2L/3],  \qquad\mbox{and} \\
B_2'(t)&=B_2(t+2L/3)\,,\qquad \qquad \qquad \qquad \qquad \qquad \quad \,\,\, \mbox{ for } t\in [0, R-2L/3]\,.
\end{align*}
Then we show that the joint law of 
\[(WB_1(2L/3), B_2(2L/3), B_1'(\cdot)-B_1'(0), B_2'(\cdot)-B_2'(0))\ll \lambda^2 \times \mu^2\,.\] 
This easily follows from Step 2, as $B_2'(\cdot)-B_2'(0)=B_2(2L/3 +\cdot)-B_2(2L/3)$ and $B_1'(\cdot)-B_1'(0)$ is a standard Brownian motion independent of $B_1(s), 0\le s\le 2L/3$, and $B_2$ and hence independent of $WB_1(2L/3), B_2(2L/3)$ and $B_2'(\cdot)-B_2'(0)$.

\medskip
\noindent{\textbf {Step 4.}} Finally we prove the corollary. With $B_1', B_2'$ as defined in Step 3, if $\mathrm{W}B'_1$ denotes the top line of the Pitman transform of $B_1'$ off $B_2'$ as defined in \eqref{e:pitmantrans}, then
\begin{equation}\label{e:defb1prime}
\mathrm{W}B_1(t+2L/3)=\mathrm{W}B'_1(t),\quad \mbox{for } t\ge 0\,.
\end{equation}
By Step 3, $B_1',B_2'$ are absolutely continuous with respect to two independent Brownian motions starting from $\lambda^2$. Conditioning on the start points, and using Lemma \ref{l:case2}, we have the result using \eqref{e:defb1prime}.
\end{proof}


Finally we have the following theorem.

 \begin{theorem}\label{l:condabscont}For any $r\in (0,y_0)$, let $\mu_r$ denote the law of a Brownian motion starting from $(0,0)$ with diffusion parameter $2$ restricted to $[r,y_0]$.
 For any $k\in \N$ and $\mathbf{g}=(g_1,g_2,\ldots,g_k)\in \R^k$, let $\xi_{k,\mathbf{g},r}$ denote the law of
\begin{equation}\label{e:opt}
\H_{k,\mathbf{g}}(y):=\max_{1\le \l\le k}(g_\l+\B[(0,\l)\to (y,1)])\,,
\end{equation}
 restricted to $[r,y_0]$, where $\B_1,\B_2,\ldots$ are independent standard Brownian motions with diffusion parameter $2$ and $\B=(\B_1,\B_2,\ldots)$. Then for all $k\in \N$, all $\mathbf{g}=(g_1,g_2,\ldots,g_k)\in \R^k$ and all $r\in (0,y_0)$,
 \[\xi_{k,\mathbf{g},r}\ll\mu_r\,.\]
 \end{theorem}
\begin{proof} We prove this by induction on $k$.

Case $k=1$. Fix $g_1\in \R$ and $r\in (0,y_0)$. Then
 \[\H_{1,g_1}(y)=g_1+\B_1(y)\,.\]
  Since $\H_{1,g_1}(r)$ is absolutely continuous with respect to the law of $\B_1(r)$, $\xi_{1,g_1,r}\ll \mu_r$.

Case $k=2$. Fix $\mathbf{g}=(g_1, g_2)\in \R^2$ and $r\in (0,y_0)$. Then using Lemma \ref{l:case2} with $L=r, R=y_0, b_1=g_1$ and $b_2=g_2$ and the Pitman transform formula in \eqref{e:ptr}, we have $\xi_{2,\mathbf{g},r}\ll \mu_r$.

Case $k\geq 3$. The idea here is to replace the last $k-1$ lines by a single line, using the metric composition law and the induction hypothesis. Then the $k$ line ensemble will reduce to the previous case with two lines. Now we make this precise.

We assume that $\xi_{k-1,\mathbf{g},r}\ll\mu_r$ for all $\mathbf{g}\in \R^{k-1}$ and all $r\in (0,y_0)$. We want to prove it for $k$. To this end, fix any $\mathbf{g}=(g_1,g_2,\ldots,g_k)\in \R^k$ and any $r\in (0,y_0)$. Let $\mathbf{g'}=(g_2,\ldots,g_{k})\in \R^{k-1}$,  and
\[\H'_{k-1,\mathbf{g'}}(y):=\max_{2\le \l\le k}(g_\l+\B[(0,\l)\to (y,2)])\,,
\]
for all $y\geq 0$.
Then clearly
\begin{itemize}
\item $\H'_{k-1,\mathbf{g'}}(0)=\max_{2\le \l \le k} g_\l\,,$
\item $\H'_{k-1,\mathbf{g'}}(\cdot)\overset{d}{=}\H_{k-1,\mathbf{g'}}(\cdot)\,,$ and
\item $\H'_{k-1,\mathbf{g'}}$ is independent of $\B_1$.
\end{itemize}
  Let
\begin{equation}\label{e:defL}
\mathscr L=(\B_1,\H'_{k-1,\mathbf{g'}})\,.
\end{equation}

 Since for all $2\leq \l\le k$, by metric composition law, Property \ref{p:metrcomp},
\[\B[(0,\l)\to (y,1)]=\sup_{0\leq t\leq y}(\B[(0,\l)\to (t,2)]+\B_1(y)-\B_1(t))\,,\]
hence
\begin{eqnarray*}
\max_{2\le \l\le k}(g_\l+\B[(0,\l)\to (y,1)])&=&\max_{2\le \l\le k}\{g_\l+\sup_{0\leq t\leq y}(\B[(0,\l)\to (t,2)]+\B_1(y)-\B_1(t))\}\\
&=&\sup_{0\leq t\leq y}\max_{2\le \l\le k}\{g_\l+\B[(0,\l)\to (t,2)]+\B_1(y)-\B_1(t)\}\\
&=&\sup_{0\leq t\leq y}\{\H'_{k-1,\mathbf{g'}}(t)+\B_1(y)-\B_1(t)\}
\end{eqnarray*}
now by definition of $\H'$ this equals
\begin{align*}
\H'_{k-1,\mathbf{g'}}(0)&+\sup_{0\leq t\leq y}\{\H'_{k-1,\mathbf{g'}}(t)-\H'_{k-1,\mathbf{g'}}(0)+\B_1(y)-\B_1(t)\}\\
&= \H'_{k-1,\mathbf{g'}}(0)+\mathscr L[(0,2)\to (y,1)]\,,
\end{align*}
where $\mathscr L$ is as defined in \eqref{e:defL}. Then
\begin{eqnarray}\nonumber
\H_{k,\mathbf{g}}(y)&=&\max_{1\le \l\le k}(g_\l+\B[(0,\l)\to (y,1)])\nonumber\\
&=&\max\{g_1+\B_1(y),\max_{2\leq \l\leq k}(g_\l+\B[(0,\l)\to (y,1)])\}
\end{eqnarray}
Since $\B_1(y)=\mathscr L[(0,1)\to (y,1)]$ we can write
\begin{eqnarray}
\H_{k,\mathbf{g}}(y)&=& \max\{g_1+\mathscr L[(0,1)\to (y,1)],\H'_{k-1,\mathbf{g'}}(0)+\mathscr L[(0,2)\to (y,1)]\}\nonumber\\
&=& \max_{\l=1,2}\{\tilde{g}_\l+\mathscr L[(0,\l)\to (y,1)]\}\,,\label{e:reducedcase2}
\end{eqnarray}
with $\tilde{g}_1=g_1$ and $\tilde{g}_2=H'_{k-1,\mathbf{g'}}(0)=\max_{2\le i \le k} g_i$.

Since $\mathscr L=(\B_1,\H'_{k-1,\mathbf{g'}})$, using the Pitman transform formula in \eqref{e:ptr}, we realize from \eqref{e:reducedcase2} that $\H_{k,\mathbf{g}}$ is the top line of the Pitman transform of $\B_1+{g}_1$ with respect to $\H'_{k-1,\mathbf{g'}}$. Here $\B_1$ is a standard Brownian motion, $\B_1$ and $\H'_{k-1,\mathbf{g'}}$ are independent and the law of $\H'_{k-1,\mathbf{g'}}(\cdot)$ on $[r/2,y_0]$ is $\xi_{k-1,\mathbf{g'},r/2}$. By induction hypothesis,
\[\xi_{k-1,\mathbf{g'},r/2}\ll \mu_{r/2}\,.\]
 Hence by Corollary \ref{l:corcase2}, we get 
   \[\xi_{k,\mathbf{g},r}\ll\mu_r\,,\]
 closing the induction.
 \end{proof}

\section{Proof of Theorem \ref{t:gen}}\label{s:proof} Finally in this section we shall prove Theorem \ref{t:gen}. However, before proceeding to proving Theorem \ref{t:gen}, we will start by proving the Brownian absolute continuity result for simpler initial conditions.

To this end, first we are interested in the KPZ fixed point $h(y):=h_1(y)$, with $h_0$  a continuous function defined on $[x_1,x_2]$, for some fixed $x_1<x_2$. Thus
\begin{equation}\label{e:defh}
h(y)= \sup_{x\in [x_1,x_2]}  \left(h_0(x)+ \mathcal \S(x,y)\right)\,.
\end{equation}
We would like to show the following.
\begin{proposition}\label{t:main} Let $h_0$ be any continuous function on $[x_1,x_2]$ and $h_0(x)=-\infty$ for all $x\notin [x_1,x_2]$. Then $h$, as defined in \eqref{e:defh},
is Brownian on compacts.
\end{proposition}

\begin{proof}[Proof of Proposition \ref{t:main}] Let $y_1<y_2$ and we consider $h(y)$ for $y\in [y_1,y_2]$. Because of translation invariance of $\S(\cdot,\cdot)$, and extending the intervals if necessary, we can assume without loss of generality that $x_1=y_1=1$ and $x_2=x_0$ and $y_2=y_0$ for some $x_0,y_0>1$.

From \eqref{e:defh}, using the continuity of $h_0(x)$ and $\S(x,y)$ as a function in $x$, we get for all $y\in[1,y_0]$,
\[h(y)=\sup_{x\in [1,x_0]}(h_0(x)+\S(x,y))=\sup_{x\in [1,x_0]\cap \Q}(h_0(x)+\S(x,y))\,.\]
Now applying Lemma \ref{l:decomp}, with $L_0$ as defined in that Lemma, we have almost surely for all $y\in [1,y_0]$,
\begin{eqnarray}\label{e:decomp}
h(y)&=&\sup_{x\in [1,x_0]\cap \Q}(h_0(x)+\max_{\l\le L_0}(\A[x\to(0,\l)]+\A[(0,\l)\to (y,1)])\nonumber\\
&=&\max_{\l\le L_0}(G_\l+\A[(0,\l)\to (y,1)])\,,
\end{eqnarray}
where
\begin{equation}\label{e:defg}
G_\l:= \sup_{x\in[1,x_0]\cap \Q} (h_0(x)+\mathcal A[x\to(0,\l)])\,.
\end{equation}

Observe that for each $\l\in \N$
\begin{itemize}
\item $G_\l<\infty$ almost surely. Indeed, by Lemma \ref{l:boundA}, almost surely for all $x\in \Q^+$, $\A[x\to (0,\l)]\leq \S(x,0)$. Hence
\[G_\l\leq \sup_{x\in[1,x_0]\cap \Q} (h_0(x)+\mathcal \S(x,0))\leq \sup_{x\in[1,x_0]} (h_0(x)+\mathcal \S(x,0))\,.\]
Since $h_0(x)+\S(x,0)$ is continuous on $[1,x_0]$, $G_\l<\infty$ almost surely.
\item $G_\l\in \F_-$ where
\[\mathcal F_-=\sigma\{\A_i(x): x\leq 0, i=1,2,\ldots\}\,.\]
This follows directly from Lemma \ref{insigma}.
\end{itemize}
 Now, for any $k\in \N$  let $\xi_{k}$ denote the law of
\begin{equation}\label{e:opt}
\mathcal H_{k}(y):=\max_{1\le \l\le k}(G_\l+\A[(0,\l)\to (y,1)])\,,
\end{equation}
restricted to $[1,y_0]$, where $G_\l$ for $\l\in \N$ are as defined in \eqref{e:defg}. Then we first show that $\xi_k\ll \mu$, where $\mu$ denotes the law of a standard Brownian motion with diffusion parameter $2$ on $[1,y_0]$.

To this end, we first make the following observation. For any $k\in \N$, let $\F_k$ denote the sigma algebra generated by the set
 \[\{\A_i(x): (i,x)\notin\{1,2,\ldots,k\}\times(0,y_0+1)\}\,.\]
 Using the Brownian Gibbs property, the law of the top $k$ lines of the Airy line ensemble between $[0,y_0+1]$ given $\F_k$ is that of $k$ independent Brownian bridges $B_1,B_2,\ldots, B_k:[0,y_0+1]\mapsto \R$ with diffusion parameter $2$ such that $B_i(0)=\A_i(0)$ and $B_i(y_0+1)=\A_i(y_0+1)$, for all $i=1,2,\ldots, k$, conditioned not to intersect each other and the line $\A_{k+1}$. Thus, the law of the top $k$ lines of the Airy line ensemble between $[0,y_0]$ given $\F_k$ is absolutely continuous with respect to that of $k$ independent Brownian motions with diffusion parameter $2$ starting from $\A_1(0)>\A_2(0)>\ldots >\A_k(0)$. So the conditional distribution of $\A_i(\cdot)-\A_i(0):[0,y_0]\mapsto\R$ for $i=1,2,\ldots,k$ given $\F_k$ is absolutely continuous with respect to $k$ independent standard Brownian motions $\B_1,\B_2,\ldots, \B_k$ with diffusion parameter $2$.

 Hence, conditioning on $\F_k$, using the observation that $G_\l\in \F_- \subseteq \F_k$, and Theorem \ref{l:condabscont}, we have
\[\xi_k\ll\mu\,.\]

Now for any measurable set $A\in \C[1,y_0]$ with $\mu(A)=0$, we have $\xi_k(A)=0$ since $\xi_k\ll\mu$, for all $k\in \N$. Hence
\begin{eqnarray*}
\P(h(y)\in A)&=&\sum_{k=1}^\infty\P(\{h(y)\in A\} \cap \{L_0=k\})\\
&=&\sum_{k=1}^\infty\P(\{\mathcal H_{k}(y)\in A\} \cap \{L_0=k\})\leq \sum_{k=1}^\infty \xi_k(A)=0\,.
\end{eqnarray*}
Hence the law of $h(y)$ on $[1,y_0]$ is absolutely continuous with respect to $\mu$.

Hence the distribution of $h(y)-h(0)$ on $[1,y_0]$ is absolutely continuous with respect to a Brownian motion starting from $(1,0)$ with diffusion parameter $2$ on $[1,y_0]$.
\end{proof}

The next proposition shows the Brownian absolute continuity for any continuous $1$-finitary initial condition $h_0$.
\begin{proposition}\label{t:cont}For any continuous $1$-finitary initial condition $h_0$ function, the random function
\[h(y)= \sup_{x\in \R}  \left(h_0(x)+ \mathcal \S(x,y)\right)\]
is Brownian on compacts.
\end{proposition}

\begin{proof}[Proof of Proposition \ref{t:cont}]Let $y_1<y_2$ and we consider $h(y)$ for $y\in [y_1,y_2]$. Let $L$ and $R$ be as in Proposition \ref{l:conditionh0} for $y\in[y_1,y_2]$. Then
\[h(y):= \sup_{x\in \R}  \left(h_0(x)+ \mathcal \S(x,y)\right)=\sup_{x\in [L,R]}  \left(h_0(x)+ \mathcal \S(x,y)\right)\,,\]
for all $y\in [y_1,y_2]$. Now for all $n\in \N$, let
\[\mathscr H_n(y):=\sup_{x\in [-n,n]}  \left(h_0(x)+ \mathcal \S(x,y)\right)\,,\]
and let $\gamma_n$ denote the law of $\mathscr H_n(y)-\mathscr H_n(y_1)$ restricted to $[y_1,y_2]$. Let $\mu$ be the law of a Brownian motion starting from $(y_1,0)$ with diffusion parameter $2$ on $[y_1,y_2]$. Then by Proposition \ref{t:main}, $\gamma_n\ll\mu$ for all $n\in \N$. Now for any measurable set $A\in \C_0[y_1,y_2]$ with $\mu(A)=0$, we have $\gamma_n(A)=0$ since $\gamma_n\ll\mu$, for all $n\in \N$. Hence
\begin{eqnarray*}
\P(h(y)-h(y_1)\in A)&\leq& \sum_{n=1}^\infty\P\left(\{h(y)-h(y_1)\in A\} \cap \{[L,R]\subseteq [-n,n]\}\right)\\
&=&\sum_{n=1}^\infty\P\left(\{\mathscr H_{n}(y)-\mathscr H_{n}(y_1)\in A\} \cap \{[L,R]\subseteq [-n,n]\}\right)\\
&\leq& \sum_{k=1}^\infty \gamma_n(A)=0\,.
\end{eqnarray*}
Hence the law of $h(y)-h(y_1)$ on $[y_1,y_2]$ is absolutely continuous with respect to $\mu$.
\end{proof}

Finally we are ready to prove Theorem\ref{t:gen}, the main result of this paper.

\begin{remark}
Observe that since $\cL(x, 0; y, t)=\S_{t^{1/3}}(x,y)=t^{1/3}\S(t^{-2/3}x,t^{-2/3}y)$, it suffices to prove Theorem \ref{t:gen} for $t=1$. For $t=1$, we have
\begin{equation}\label{e:defh2}
h(y):=h_1(y)= \sup_{x\in \R}  \left(h_0(x)+ \mathcal \S(x,y)\right)\,.
\end{equation}
\end{remark}

\begin{proof}[Proof of Theorem \ref{t:gen}] Without loss of generality, assume $t=1$. Let $y_1<y_2$ and we consider $h(y)$ for $y\in [y_1,y_2]$. First we assume that $h_0$ is defined on $[x_1,x_2]$ for some $x_1<x_2$  (that is $h_0(x)=-\infty$ for all $x\notin [x_1,x_2]$). Since $h_0$ is $1$-finitary and defined on $[x_1,x_2]$, hence $h_0$ is bounded above. Let $h_0(x)\leq B$ for all $x\in [x_1,x_2]$.

Now by the metric composition law of the directed landscape, we have
\[\cL(x,0;y,1)=\sup_{z\in \R}(\cL(x,0;z,1/2)+\cL(z,1/2;y,1))\,.\]
Thus from \eqref{e:defh2},
\begin{eqnarray}\label{e:split}
h(y)&=& \sup_{x\in [x_1,x_2]}  \left(h_0(x)+ \cL(x,0;y,1)\right)\\
&=&\sup_{x\in [x_1,x_2],z\in \R}  \left(h_0(x)+ \cL(x,0;z,1/2)+\cL(z,1/2;y,1)\right)\nonumber\\
&=&\sup_{z\in \R}\left(\mathcal U(z)+\cL(z,1/2;y,1)\right)\,,
\end{eqnarray}
where
\[\mathcal U(z):=\sup_{x\in [x_1,x_2]}(h_0(x)+ \cL(x,0;z,1/2))\,,\]
for all $z\in \R$. Because of independent increment property of directed landscape, we have $\mathcal U(\cdot)$ and $\cL(\cdot,1/2;\cdot,1)$ are independent.
Since
\[|\mathcal U(z_1)-\mathcal U(z_1)|\leq \sup_{x\in [x_1,x_2]}|\cL(x,0;z_1,1/2)-\cL(x,0;z_2,1/2)|\,,\]
and $\cL(\cdot,0;\cdot,1/2)$ is continuous, hence $\mathcal U(\cdot)$ is continuous. Moreover, we claim that $\mathcal U$ is $1/2$-finitary almost surely. Indeed, for all $z\in \R$,
\begin{equation*}
\mathcal U(z)\leq c+\sup_{x\in [x_1,x_2]}\cL(x,0;z,1/2)\leq C+C|z|^{1/3}\,,
\end{equation*}
for some constant $c$ and some random constant $C$ by \citeDOV, Corollary $10.7$.

Thus, conditioning on $\cL(\cdot,0;\cdot,1/2)$ and applying Proposition \ref{t:cont}, from \eqref{e:split} (using the scaling property of the directed landscape that $\cL(z,1/2;y,1)\overset{d}{=}2^{-1/3}\S(2^{2/3}x,2^{2/3}y)$), we have the law of $h(y)-h(y_1)$ restricted to $[y_1,y_2]$ is absolutely continuous with respect to $\mu$, the law of a Brownian motion starting from $(y_1,0)$ with diffusion parameter $2$ on $[y_1,y_2]$.

Now we prove Theorem \ref{t:gen} for any $1$-finitary initial condition $h_0$. By Lemma $6.1$, we get two random variables $L<R$ such that
\[h(y):= \sup_{x\in \R}  \left(h_0(x)+ \mathcal \S(x,y)\right)=\sup_{x\in [L,R]}  \left(h_0(x)+ \mathcal \S(x,y)\right)\,,\]
for all $y\in [y_1,y_2]$. Now the rest of the proof follows exactly as in that of Proposition \ref{t:cont}, by using the corresponding result of absolute continuity for compactly defined $h_0$ proved above.
\end{proof}

\section{Finitary initial conditions}\label{s:last}

We check that a $t_0$-finitary initial condition is necessary and sufficient for a.s.\ finiteness of $h_t(y)$ for all $y\in \mathbb R, t\in (0,t_0]$. Morover, for finitary initial conditions, the optimization happens in a compact interval.

\begin{proposition}\label{l:conditionh0}
For any $t_0$-finitary initial condition $h_0$, a.s.\ $h_t(y)$ is finite for all $y\in \R, t\in(0,t_0]$. Moreover given a compact $K\subset \mathbb R \times (0,t_0]$ there exist random variables $L,R$ so that
$$h_t(y) =\sup_{x\in[L,R]} (h_0(x) + \mathcal L(x,0;y,t))\,.
$$
for all $(y,t)\in K$.

Conversely, for any $h_0$ that is not $t_0$-finitary and $h_0\not\equiv -\infty$, there exists an infinite interval $A$ so that a.s. $h_{t_0}(y)=\infty$ for all $y\in A$.
\end{proposition}
\begin{proof}


By Corollary 10.7 of \citeDOV\  with a random constant $C$ the directed landscape satisfies
$$
|\mathcal L(x,0;y,t) + (x-y)^2/t| \le C(1+|x|^{1/5}+|y|^{1/5}+|t|^{1/5})(t^{1/3}\vee \log_+^{4/3} (1/t))\,.
$$
Recall that $h_t(y) =\sup_x (h_0(x) + \mathcal L(x,0;y,t))$.
 With  some random continuous functions $C$ of $(y,t)\in\mathbb R\times(0,t_0]$ we have
\begin{eqnarray}\label{e:lastbound}
-C+\sup_x \left(h_0(x)-\frac{x^2}{t} +\frac{2xy}t - C|x|^{1/5}\right) & \le& h_t(y)\notag\\
& \le& C+\sup_x \left(h_0(x)-\frac{x^2}{t} +\frac{2xy}{t} + C|x|^{1/5}\right).
\end{eqnarray}

The upper bound implies that when $h_0$ is a $t$-finitary initial condition, then $h_t(y)<\infty$. Indeed, for such $h_0$, the upper bound converges to $-\infty$ as $|x|\to \infty$ uniformly over $(y,t)$ in fixed compact subsets of $\mathbb R\times(0,t_0]$.

Now we assume that $h_0$ is not $t_0$-finitary. If $h_0$ is not bounded above in some compact interval, then by the lower bound in \eqref{e:lastbound}, $h_{t_0}\equiv \infty$. On the other hand, when $h_0(x)-x^2/t_0 > a|x|$ along some sequence $x_n \to \pm \infty$, the expression in the supremum in the lower bound in \eqref{e:lastbound} converges to $\infty$ along this sequence whenever $\mp2y/t_0<a$.
\end{proof}

\bigskip

\noindent {\bf Acknowledgments.}  B.V. was supported by the Canada Research Chair program, the NSERC Discovery Accelerator grant and the MTA Momentum Random Spectra research group. The authors thank Duncan Dauvergne for fruitful discussions. \sourav{The authors thank the anonymous referee whose careful reading and detailed comments helped improve the paper.}

\bibliographystyle{amsplain}

\end{document}